\documentclass[11pt,reqno]{amsart}

\usepackage{amssymb}
\usepackage{amsmath}
\usepackage{amsthm}
\usepackage{latexsym}
\usepackage{amsfonts}
\usepackage{mathrsfs}
\usepackage{titlesec}
\usepackage{color}
\usepackage{hyperref}

\newtheoremstyle{neu_thm}% name
{13pt}       % Space above
{8pt}      % Space below
{\itshape}  % Body font
{}          % Indent amount (empty = no indent, \parindent = para indent)
{\bfseries} % Thm head font
{.}         % Punctuation after thm head
{.5em}      % Space after thm head: " " = normal interword space;
{}          % Thm head spec (can be left empty, meaning 'normal') 

\newtheoremstyle{neu_defn}% name
{13pt}       % Space above
{8pt}      % Space below
{}  % Body font
{}          % Indent amount (empty = no indent, \parindent = para indent)
{\bfseries} % Thm head font
{.}         % Punctuation after thm head
{.5em}      % Space after thm head: " " = normal interword space;
{}          % Thm head spec (can be left empty, meaning 'normal')

\theoremstyle{neu_thm}
\newtheorem{thm}{Theorem}[section]
\newtheorem{cor}[thm]{Corollary}
\newtheorem{lem}[thm]{Lemma}

\theoremstyle{neu_defn}

\newtheorem{rem}[thm]{Remark}

\titleformat{\section}{\normalfont\bfseries\centering}{\thesection.}{.25em}{}
\titleformat{\subsection}{\normalfont\bfseries}{\thesubsection.}{.25em}{}
\titlespacing{\section}{0pt}{*5}{*1.5}
\titlespacing{\subsection}{0pt}{*4}{*0.5}
\numberwithin{equation}{section}
\allowdisplaybreaks
\newcommand{\paragraf}{\textsection}
\renewcommand{\emptyset}{\varnothing}

\newcommand{\C}{\ensuremath{\mathbb C}}    % Komplexe Zahlen
\newcommand{\calF}{\mathcal F}         
         
\newcommand{\calH}{\mathcal H}

\newcommand{\calK}{\mathcal K}

\newcommand{\la}{\lambda}
\newcommand{\veps}{\varepsilon}

\newcommand{\linspan}{\operatorname{span}}

\newcommand{\ran}{\operatorname{ran}}

\newcommand{\ol}{\overline}

\definecolor{darkgreen}{rgb}{0,0.6,0.1}

\newcommand{\ind}{\operatorname{ind}}
\newcommand{\dist}{\operatorname{dist}}
\newcommand{\defi}{\operatorname{def}}
\newcommand{\nuli}{\operatorname{nul}}
\newcommand{\rank}{\operatorname{rank}}
\newcommand{\rmref}[1]{{\rm\ref{#1}}}

\begin{document}
%%%%%%%%%%%%%%%%%%%%%%%%%%%%%%%%%%%%%%%%%%%%%%%%%%%%%%%%%%%%%%%%%%%%%%%%%%%%%
%%%
%%%  HEAD OF PAPER
%%%
\title[]{The Effect of Perturbations of Linear Operators on Their Polar Decomposition}

\author{Richard Duong}
\address{Institut f\"ur Mathematik, Technische Universit\"at Berlin, Stra\ss e des 17.\ Juni 136, 10623 Berlin, Germany}
\email{richard\_duong@gmx.de}

\author{Friedrich Philipp}
\address{Departamento de Matem\'atica, Facultad de Ciencias Exactas y Naturales, Universidad de Buenos Aires, Ciudad Universitaria, Pabell\'on I, 1428 Buenos Aires, Argentina}
\email{fmphilipp@dm.uba.ar}
\urladdr{http://cms.dm.uba.ar/Members/fmphilipp/}

%%%%%%%%%%%%%%%%%%%%%%%%%%%%      ABSTRACT      %%%%%%%%%%%%%%%%%%%%%%%%%%%%%
\begin{abstract}
The effect of matrix perturbations on the polar decomposition has been studied by several authors and various results are known. However, for operators between infinite-dimensional spaces the problem has not been considered so far. Here, we prove in particular that the partial isometry in the polar decomposition of an operator is stable under perturbations, given that kernel and range of original and perturbed operator satisfy a certain condition. In the matrix case, this condition is weaker than the usually imposed equal-rank condition. It includes the case of semi-Fredholm operators with agreeing nullities and deficiencies, respectively. In addition, we prove a similar perturbation result where the ranges or the kernels of the two operators are assumed to be sufficiently close to each other in the gap metric.
\end{abstract}
%%%%%%%%%%%%%%%%%%%%%%%%%%%%%%%%%%%%%%%%%%%%%%%%%%%%%%%%%%%%%%%%%%%%%%%%%%%%%



%\thanks{No thanks yet}

\maketitle
\thispagestyle{empty}

\section{Introduction}
The polar decomposition of a bounded linear operator $A\in L(\calH,\calK)$ between two Hilbert spaces $\calH$ and $\calK$ is a representation $A = Q|A|$ of $A$ where $Q\in L(\calH,\calK)$ is a partial isometry and $|A| = (A^*A)^{1/2}$. It is well known that $Q$ in this decomposition is unique if it is required that its kernel coincides with that of $A$. The operator $Q$ will be called the {\em angular factor} of $A$.

Studies on the behaviour of the polar decomposition of matrices under perturbations can be traced back to the paper \cite{h} by N.J. Higham from 1986 where the author in particular finds estimates on the Frobenius norm difference $\|Q_1 - Q_2\|_F$ of the angular factors $Q_1$ and $Q_2$ of two invertible square matrices $A_1$ and $A_2$, which are close to each other. He then introduces an algorithm which computes the polar decomposition of a matrix and uses his result in an error analysis for the algorithm. Such algorithms enjoy applications in, e.g., psychometrics or aerospace computations. However, the perturbation result was far from being optimal in various aspects, and many papers written by several authors followed (see, e.g., \cite{b,ma,rcli1,rcli2,cg,wlisun,cls,cl,wli,hmz} (in chronological order)) -- generalizing the conditions and consecutively improving the perturbation bound. Also rectangular matrices and other norms than the Frobenius norm were considered. The perturbation results were all of the following form: {\it Let $A_1,A_2\in\C^{m\times n}$ be matrices of the same rank. Then for their angular factors $Q_1$ and $Q_2$ we have}
\begin{equation}\label{e:general}
\|Q_1 - Q_2\|_*\,\le\,C\|A_1 - A_2\|_*.
\end{equation}
Here, $\|\cdot\|_*$ is a general unitarily invariant norm, the Frobenius norm or the spectral norm, and the constant $C$ depends on the smallest positive singular values of $A_1$ and $A_2$ (and does not tend to infinity as $A_1$ tends to $A_2$). Often, $C$ is of the form $D/(\sigma_1 + \sigma_2)$, where $D$ is a fixed number and $\sigma_j$ the smallest positive singular value of $A_j$.

In what follows, the requirement that $\rank(A_1) = \rank(A_2)$ will be referred to as the {\em rank condition}. It quickly becomes clear that in order to have an inequality as in \eqref{e:general} it is necessary to impose such an assumption on the matrices. For example, consider the matrices $A_1 = 0$ and $A_2 = \veps I$, $\veps > 0$, which do not satisfy the rank condition. Then $Q_1 = 0$ and $Q_2 = I$ and thus $\|A_1 - A_2\|_2 = \veps\to 0$ as $\veps\to 0$, but $\|Q_1 - Q_2\|_2 = 1$. However, as we shall see, the rank condition can be relaxed such that \eqref{e:general} still holds -- at least in the case of the spectral norm.

It is evident that the rank condition on the matrices in the results above will not make sense for linear operators between infinite-dimensional spaces. In this paper we consider the question how this condition can be replaced by a new criterion which is also meaningful in the infinite-dimensional situation and under which a perturbation relation as in \eqref{e:general} still holds. Hereby, we restrict ourselves to the operator norm (which is the spectral norm in the finite-dimensional setting). The following theorem is a special case of our main result Theorem \ref{t:main} and generalizes the matrix perturbation results to the case of bounded semi-Fredholm operators between possibly infinite-dimensional Hilbert spaces. Recall that nullity $\nuli(A)$ and deficiency $\defi(A)$ of a bounded linear operator $A$ are defined via the dimension of its kernel and the codimension of its range, respectively, and that the operator $A$ is called semi-Fredholm if its range is closed and one of these quantities is finite.

\begin{thm}\label{t:fredholm}
Let $\calH$ and $\calK$ be Hilbert spaces and let $A_1,A_2\in L(\calH,\calK)\setminus\{0\}$ be semi-Fredholm such that $\nuli(A_1) = \nuli(A_2)$ and $\defi(A_1) = \defi(A_2)$. If $A_j = Q_j|A_j|$ is the polar decomposition of $A_j$, $j=1,2$, then
\begin{equation}\label{e:bound}
\|Q_1 - Q_2\|\,\le\,\frac{4}{\sigma_1 + \sigma_2}\,\|A_1 - A_2\|,
\end{equation}
where $\sigma_j$ denotes the smallest positive spectral value of $|A_j| = (A_j^*A_j)^{1/2}$, $j=1,2$.
\end{thm}

Clearly, in the finite-dimensional situation, the condition on $A_1$ and $A_2$ in Theorem \ref{t:fredholm} coincides with the rank condition $\rank(A_1) = \rank(A_2)$. However, we would like to emphasize that the condition in our main result Theorem \ref{t:main} is weaker than the rank condition in the matrix case. An illustrative example is given in Remark \ref{r:vergleich}. Hence, Theorem \ref{t:main} in particular extends the above-mentioned results for matrices.

We shall also prove the following theorem which shows that a similar perturbation bound as in \eqref{e:bound} holds for closed range operators whose ranges or kernels are ``not too far apart from each other'' in the gap metric $\hat\delta$ (to be introduced in Subsection \ref{ss:gap}).

\begin{thm}\label{t:cr_gap}
Let $\calH$ and $\calK$ be Hilbert spaces and let $A_1,A_2\in L(\calH,\calK)\setminus\{0\}$ be operators with closed ranges. If $A_j = Q_j|A_j|$ is the polar decomposition of $A_j$, $j=1,2$, then
\begin{equation}\label{e:side_estimate}
\|Q_1 - Q_2\|\,\le\,\left(\frac{3}{\sigma_1 + \sigma_2} + \frac{1}{\min\{\sigma_1,\sigma_2\}}\right) \|A_1 - A_2\|,
\end{equation}
where $\sigma_j$ denotes the smallest positive spectral value of $|A_j| = (A_j^*A_j)^{1/2}$, $j=1,2$. If the gap between the ranges or the kernels is smaller than one, i.e.,
\begin{equation}\label{e:gaps}
\hat\delta(R(A_1),R(A_2)) < 1\qquad\text{or}\qquad\hat\delta(N(A_1),N(A_2)) < 1,
\end{equation}
then ``$\min\!$'' in \eqref{e:side_estimate} can be replaced by ``$\max\!$'' so that, in this case,
\begin{equation}\label{e:mitgap}
\|Q_1 - Q_2\|\,\le\,\frac{5}{\sigma_1 + \sigma_2}\,\|A_1 - A_2\|.
\end{equation}
\end{thm}

Theorem \ref{t:cr_gap} is not a generalization of the matrix perturbation results. Indeed, if $\calH$ and $\calK$ are finite-dimensional, then \eqref{e:gaps} implies $\rank(A_1) = \rank(A_2)$, but the converse is false. However, if, in addition to the rank condition, we assume that $A_1$ and $A_2$ are small perturbations of each other, then \eqref{e:gaps} holds (cf.\ Remark \ref{r:smallpert}). Hence, Theorem \ref{t:cr_gap} is a generalization of the matrix results for small perturbations.

The paper is organized as follows. In Section 2 we collect several known facts and auxiliary statements from operator theory and Hilbert space theory which will be used in the proofs of the perturbation results. In this process, we also define and discuss the {\em gap difference} $\Delta(V,W)$ between subspaces $V$ and $W$ which seems to be uncommon but enters in the conditions of the main theorem, Theorem \ref{t:main}. This theorem is then formulated and proved in Section 3 where we also prove Theorem \ref{t:cr_gap} and Theorem \ref{t:fredholm}, as a corollary of Theorem \ref{t:main}.

\section{Preparations}
In this section we fix the notation we will use and recall a few facts from operator theory and Hilbert space theory which we shall utilize in the proof of the main theorem.

Throughout the paper, $\calH$ and $\calK$ always denote complex Hilbert spaces. By $L(\calH,\calK)$ we denote the set of all bounded linear operators mapping from $\calH$ to $\calK$. As usual, we write $L(\calH) := L(\calH,\calH)$. The operator norm of $S\in L(\calH,\calK)$ will be denoted by $\|S\|$. The spectrum of an operator $T\in L(\calH)$ is denoted by $\sigma(T)$, i.e.,
$$
\sigma(T) = \{\la\in\C\,|\, T - \la I : \calH\to\calH\text{ is not bijective}\}.
$$
For $S\in L(\calH,\calK)$ we denote the range and the kernel of $S$ by $R(S)$ and $N(S)$, respectively. Furthermore, we set $|S| := (S^*S)^{1/2}$. If $M\subset\calH$ is a closed subspace, $P_M$ denotes the orthogonal projection onto $M$. The restriction of $S\in L(\calH,\calK)$ to $M$ will be denoted by $S|M$. The notion ``$\dim$'' always refers to the Hilbert space dimension which is a finite or infinite cardinal number.

\subsection{Prerequisites from Operator Theory}
Theorem \ref{t:fredholm} in the introduction involves the ``smallest positive spectral value'' of a positive selfadjoint operator with closed range. The next lemma shows in particular that this value in fact exists. For a proof we refer to \cite{kn}.

\begin{lem}\label{l:ranclosed}
Let $A\in L(\calH,\calK)$, $A\neq 0$, and assume that $R(A)$ is closed. Then we have
\begin{equation}\label{e:ran}
R((A^*A)^{1/2}) = R(A^*A) = R(A^*)
\end{equation}
as well as
\begin{equation}\label{e:sigma}
\inf\left(\sigma(|A|)\setminus\{0\}\right) = \left\|\left(A\big|N(A)^\perp\right)^{-1}\right\|^{-1} = \left\|\left(|A|\big|N(A)^\perp\right)^{-1}\right\|^{-1} > 0,
\end{equation}
where $A|N(A)^\perp$ is considered as an operator from $N(A)^\perp$ to $R(A)$.
\end{lem}

We denote the value in \eqref{e:sigma} by $\sigma_A$. It is often called the {\it reduced minimum modulus} of A (see, e.g., \cite[p.\ 231]{k}). The following well known Jacobsen Lemma implies in particular that $\sigma_{A^*} = \sigma_A$.

\begin{lem}\label{l:swap}
Let $S\in L(\calH,\calK)$ and $T\in L(\calK,\calH)$. Then
$$
\sigma(ST)\setminus\{0\} = \sigma(TS)\setminus\{0\}.
$$
\end{lem}

For a proof of the following theorem we refer to the survey \cite{br} by Bhatia and Rosenthal on the {\em Sylvester equation}.

\begin{thm}\label{t:sylvester}
Let $S\in L(\calH)$ and $T\in L(\calK)$ and assume that $\sigma(T)\cap\sigma(S) = \emptyset$. Then for each $Y\in L(\calH,\calK)$, the operator equation
\begin{equation}\label{e:sylvester}
XS - TX = Y
\end{equation}
has a unique solution $X\in L(\calH,\calK)$. If $S$ and $T$ are normal such that $\sigma(T)$ lies in a circle $B_r(a)$ and $\sigma(S)$ lies outside of $B_{r+\delta}(a)$, then the solution $X$ of {\rm\eqref{e:sylvester}} satisfies
$$
\|X\|\,\le\,\frac{\|Y\|}{\delta}.
$$
\end{thm}

For a linear operator $A\in L(\calH,\calK)$ we define the {\em nullity} and {\em deficiency} of $A$ by $\nuli(A) := \dim N(A)$ and $\defi(A) := \dim R(A)^\perp$, respectively. These might be infinite cardinal numbers. Addition and subtraction of cardinal numbers $\kappa$ and $\mu$ is defined as follows:
\begin{align*}
\kappa + \mu
&:= 
\begin{cases}
\kappa + \mu &\text{if $\mu$ and $\kappa$ are finite}\\
\max\{\kappa,\mu\} &\text{otherwise}
\end{cases},\\
\kappa - \mu
&:=
\begin{cases}
\kappa - \mu &\text{if $\mu$ and $\kappa$ are finite}\\
0 &\text{if $\kappa = \mu$}\\
\kappa &\text{if $\kappa > \mu$}\\
-\mu &\text{if $\kappa < \mu$}.
\end{cases}
\end{align*}
We now define the {\em index} of $A$ by
$$
\ind(A) := \dim N(A) - \dim R(A)^\perp.
$$
Recall that $A$ is called {\em semi-Fredholm} if $R(A)$ is closed and one of the numbers $\nuli(A)$ and $\defi(A)$ is finite.

\subsection{The Gap Between Subspaces}\label{ss:gap}
For two closed subspaces $V$ and $W$ of $\calH$ the {\em gap from $V$ to $W$} is defined by
$$
\delta(V,W) := \sup\left\{\|v - P_Wv\| : v\in V,\,\|v\|=1\right\} = \left\|(I - P_W)|V\right\| = \|P_{W^\perp}|V\|.
$$
It is worth noting that
\begin{equation}\label{e:gap_proj2}
\delta(V^\perp,W^\perp) = \|P_W|V^\perp\| = \|(P_W|V^\perp)^*\| = \|P_{V^\perp}|W\| = \delta(W,V).
\end{equation}
As $\delta$ is not a metric, Kato (see \cite[\paragraf IV.2]{k}) defines the {\em gap between $V$ and $W$} by
\begin{equation}\label{e:maxgap}
\hat\delta(V,W) := \max\left\{\delta(V,W),\delta(W,V)\right\}.
\end{equation}
From $\|P_Vx - P_Wx\|^2 = \|P_WP_{V^\perp}x\|^2 + \|P_{W^\perp}P_Vx\|^2$ for $x\in\calH$ it is not hard to deduce that
\begin{equation}\label{e:gap_proj}
\hat\delta(V,W) = \|P_V - P_W\|.
\end{equation}
The next lemma is well known (see, e.g., \cite[Theorem I-6.34]{k}).

\begin{lem}\label{l:gap}
Let $V$ and $W$ be closed subspaces of $\calH$. Then the following statements hold.
\begin{enumerate}
\item[{\rm (i)}]  If $\delta(V,W) < 1$, then $P_W|V\in L(V,W)$ is bounded below.
\item[{\rm (ii)}] If $\hat\delta(V,W) < 1$, then $\delta(V,W) = \delta(W,V)$, and the operators $P_W|V\in L(V,W)$ and $P_V|W\in L(W,V)$ are isomorphisms.
\end{enumerate}
\end{lem}

The following lemma can be found in \cite{mar}. However, for the sake of self-containedness we provide a proof here.

\begin{lem}\label{l:gappyhelpy}
Let $T,S\in L(\calH,\calK)$, $T\neq 0$, and assume that $R(T)$ is closed. Then
$$
\delta\left(R(T),\ol{R(S)}\right)\,\le\,\frac{\|T - S\|}{\sigma_T}.
$$
In particular, if both $R(T)$ and $R(S)$ are closed, then
\begin{equation}\label{e:gap_min}
\hat\delta\left(R(T),R(S)\right)\,\le\,\frac{\|T - S\|}{\min\{\sigma_S,\sigma_T\}}.
\end{equation}
If $\hat\delta\left(R(T),R(S)\right) < 1$, then $\min\{\sigma_S,\sigma_T\}$ in \eqref{e:gap_min} can be replaced by $\max\{\sigma_S,\sigma_T\}$.
\end{lem}
\begin{proof}
We have $\delta(V,W) = \sup\{\dist(v,W) : v\in V,\,\|v\|=1\}$. This easily follows from the definition of $\delta(V,W)$. Hence,
\begin{align*}
\delta\left(R(T),\ol{R(S)}\right)
&= \sup\left\{\inf_{z\in R(S)}\|y - z\| : y\in R(T),\,\|y\|=1\right\}\\
&= \sup\left\{\inf_{u\in\calH}\|Tx - Su\| : x\in N(T)^\perp,\,\|Tx\|=1\right\}\\
&\le \sup\left\{\|Tx - Sx\| : x\in N(T)^\perp,\,\|Tx\|=1\right\}\\
&\le \|T - S\|\cdot\sup\{\|x\| : x\in N(T)^\perp,\,\|Tx\|=1\}\\
&= \|T - S\|\cdot\sup\{\|x\| : x\in N(T)^\perp,\,\||T|x\|=1\}\\
&= \|T - S\|\cdot\sigma_T^{-1}.
\end{align*}
If also $R(S)$ is closed, then
$$
\hat\delta\left(R(T),R(S)\right)\le\max\left\{\frac{\|T - S\|}{\sigma_T},\frac{\|T - S\|}{\sigma_S}\right\} = \frac{\|T - S\|}{\min\{\sigma_S,\sigma_T\}}.
$$
Now, assume that $\hat\delta\left(R(T),R(S)\right) < 1$. Then $\delta(R(T),R(S)) = \delta(R(S),R(T))$ by Lemma \ref{l:gap}, which immediately yields the claim.
\end{proof}

In addition to the gap metric $\hat\delta$ we define the {\em gap difference} between the closed subspaces $V$ and $W$ by
$$
\Delta(V,W) := |\delta(V,W) - \delta(W,V)|.
$$
We are particularly interested in the case where $\Delta(V,W) = 0$.

\begin{lem}\label{l:Deltazero}
Let $V,W\subset\calH$ be closed subspaces of $\calH$. Then $\Delta(V,W) = 0$ holds if and only if $P_V|W\in L(W,V)$ and $P_W|V\in L(V,W)$ are both surjective or both not.
\end{lem}
\begin{proof}
Define the operators $T_{11} := P_V|W\in L(W,V)$, $T_{12} := P_V|W^\perp\in L(W^\perp,V)$, and $T_{21} := P_{V^\perp}|W\in L(W,V^\perp)$. Note that $\|T_{12}\| = \delta(V,W)$, $\|T_{21}\| = \delta(W,V)$, and $T_{11}^* = P_W|V$. We have
$$
T_{11}T_{11}^* + T_{12}T_{12}^* = P_VP_W|V + P_VP_{W^\perp}|V = P_V(P_W + P_{W^\perp})|V = P_V|V = I_V.
$$
Similarly, one sees that $T_{11}^*T_{11} + T_{21}^*T_{21} = I_W$. Thus,
\begin{align*}
\|T_{12}\|^2 &= \|T_{12}T_{12}^*\| = \|I_V - T_{11}T_{11}^*\|,\\
\|T_{21}\|^2 &= \|T_{21}^*T_{21}\| = \|I_W - T_{11}^*T_{11}\|.
\end{align*}
Since $\|T_{11}\| = \|P_V|W\|\le 1$, the operators $I_V - T_{11}T_{11}^*$ and $I_W - T_{11}^*T_{11}$ have their spectra in $[0,1]$. Hence,
\begin{align*}
\|T_{12}\|^2 &= \max\,\sigma(I_V - T_{11}T_{11}^*) = 1 - \min\,\sigma(T_{11}T_{11}^*),\\
\|T_{21}\|^2 &= \max\,\sigma(I_W - T_{11}^*T_{11}) = 1 - \min\,\sigma(T_{11}^*T_{11}).
\end{align*}
Therefore, we have $\Delta(V,W) = 0$ if and only if $T_{11}$ and $T_{11}^*$ are both surjective or both not.
\end{proof}

\begin{cor}\label{c:suff_Deltazero}
Let $V,W\subset\calH$ be closed subspaces of $\calH$. Then for $\Delta(V,W) = 0$ it is sufficient that one of the following conditions holds.
\begin{enumerate}
\item[{\rm (i)}]   $P_VW$ is not closed.
\item[{\rm (ii)}]  $P_WV$ is not closed.
\item[{\rm (iii)}] $\dim V = \dim W < \infty$.
\end{enumerate}
\end{cor}
\begin{proof}
Let $T := P_V|W\in L(W,V)$. If $\ran T = P_VW$ is not closed then the same is true for $\ran T^* = P_WV$. In particular, (i) and (ii) are equivalent. In this case, both operators $P_V|W$ and $P_W|V$ cannot be surjective. Hence, Lemma \ref{l:Deltazero} implies $\Delta(V,W) = 0$. In general, $T$ and $T^*$ are both bijective or both not. But in the case of (iii), bijectivity coincides with surjectivity so that $\Delta(V,W) = 0$ follows directly from Lemma \ref{l:Deltazero}.
\end{proof}

\subsection{The Polar Decomposition of a Bounded Operator}
Recall that every bounded operator $A\in L(\calH,\calK)$ admits a so-called {\em polar decomposition}
$$
A = Q|A|,
$$
where
\begin{itemize}
\item[{\rm (a)}] $|A| = (A^*A)^{1/2}\in L(\calH)$, $Q\in L(\calH,\calK)$,
\item[{\rm (b)}] $Q|N(Q)^\perp$ is isometric, and
\item[{\rm (c)}] $N(Q) = N(A)$.
\end{itemize}
The polar decomposition of $A$ with the properties (a)--(c) is unique (see, e.g., \cite[VI.2.7]{k}). Note that (b) implies that $R(Q)$ is closed. An operator $Q\in L(\calH,\calK)$ satisfying (b) is often called a {\em partial isometry}. As is easily seen, we have that
\begin{equation}\label{e:Q}
Q^*Q = P_{N(Q)^\perp}\quad\text{and}\quad QQ^* = P_{R(Q)}.
\end{equation}
The partial isometry in the polar decomposition of $A\in L(\calH,\calK)$ will be called the {\em angular factor} of $A$. Here, we will also denote it by $Q_A$. Hence, $A = Q_A|A|$.

\begin{lem}\label{l:polar1}
Let $Q$ be the angular factor of $A\in L(\calH,\calK)$. If $R(A)$ is closed, then
$$
R(Q) = R(A),
$$
and $Q^*$ is the angular factor of $A^*$.
\end{lem}
\begin{proof}
Since $R(A)$ is closed, Lemma \ref{l:ranclosed} implies $R(|A|) = R(A^*) = N(A)^\perp = N(Q)^\perp$. Now, the first assertion follows: $R(A) = R(Q|A|) = QR(|A|) = QN(Q)^\perp = R(Q)$. For the second claim, we observe that $A^* = |A|Q^* = Q^*(Q|A|Q^*)$. Since $Q^*$ is a partial isometry with $N(Q^*) = N(A^*)$, it suffices to prove that $Q|A|Q^* = (AA^*)^{1/2}$. But this is clear since the operator $Q|A|Q^*$ is selfadjoint, non-negative and satisfies $(Q|A|Q^*)^2 = (Q|A|)Q^*Q(|A|Q^*) = AP_{R(A^*)}A^* = AA^*$.
\end{proof}

\begin{lem}\label{l:ext}
Let $A\in L(\calH,\calK)$ have closed range such that $\ind(A) = 0$. Then there exists a unitary operator $U : \calH\to\calK$ such that
$$
A = U|A|,\quad UN(A) = R(A)^\perp,\quad U|N(A)^\perp = Q|N(A)^\perp,\;\text{ and}\quad U^*Q = P_{N(A)^\perp},
$$
where $Q$ denotes the angular factor of $A$.
\end{lem}
\begin{proof}
By assumption, we have $\dim N(A) = \dim R(A)^\perp$. Hence, there exists a unitary operator $W : N(A)\to R(A)^\perp$. Now define
$$
U := QP_{N(A)^\perp} + WP_{N(A)}.
$$
It is clear that $A = U|A|$, $UN(A) = R(A)^\perp$, and $U|N(A)^\perp = Q|N(A)^\perp$. Since $R(Q) = R(A)$ and $R(W) = R(A)^\perp$, for $x\in\calH$ we have
$$
\|Ux\|^2 = \|QP_{N(A)^\perp}x\|^2 + \|WP_{N(A)}x\|^2 = \|P_{N(A)^\perp}x\|^2 + \|P_{N(A)}x\|^2 = \|x\|^2.
$$
As $U$ is evidently surjective, $U$ is unitary. And since $R(W) = R(A)^\perp = R(Q)^\perp = N(Q^*)$, the relation \eqref{e:Q} implies
$$
Q^*U = Q^*QP_{N(A)^\perp} + Q^*WP_{N(A)} = P_{N(Q)^\perp}P_{N(A)^\perp} = P_{N(A)^\perp}.
$$
This proves the lemma.
\end{proof}

\section{Perturbation Results}
The following theorem is the main result of this paper.

\begin{thm}\label{t:main}
Let $A_1,A_2\in L(\calH,\calK)\setminus\{0\}$ have closed ranges such that $\ind(A_1) = \ind(A_2)$ and
\begin{equation}\label{e:Deltadiff}
\Delta(R(A_1),R(A_2)) = 0\qquad\text{or}\qquad\Delta(N(A_1),N(A_2)) = 0.
\end{equation}
If $A_j = Q_j|A_j|$ is the polar decomposition of $A_j$ and $\sigma_j = \sigma_{A_j}$, $j=1,2$, then
\begin{equation}\label{e:main_estimate}
\|Q_1 - Q_2\|\,\le\,\frac{4}{\sigma_1 + \sigma_2}\,\|A_1 - A_2\|.
\end{equation}
\end{thm}

\begin{rem}\label{r:vergleich}
Note that, in the finite-dimensional situation where $\dim\calH,\dim\calK < \infty$, we always have $\ind(A_1) = \ind(A_2)$ and it follows from Lemma \ref{l:Deltazero} that the rank condition $\rank A_1 = \rank A_2$ implies \eqref{e:Deltadiff}. However, the converse is false. To see this, we give a simple example: Let $\calH = \calK = \C^3$ and $A_1 := P_{\linspan\{e_1\}}$, $A_2 := P_{\linspan\{e_2,e_3\}}$, where $\{e_1,e_2,e_3\}$ denotes the standard basis of $\C^3$. Here, we have
$$
P_{N(A_1)}|N(A_2) = P_{N(A_2)}|N(A_1) = P_{R(A_1)}|R(A_2) = P_{R(A_2)}|R(A_1) = 0.
$$
Hence, \eqref{e:Deltadiff} holds due to Lemma \ref{l:Deltazero}, but $\rank(A_1)\neq\rank(A_2)$. Therefore, Theorem \ref{t:main} even generalizes the known results in the finite-dimensional case.
\end{rem}

\begin{proof}[Proof of Theorem \rmref{t:main}]
The proof is divided into two steps. In the first step, we assume that $\ind(A_1) = \ind(A_2) = 0$ and then reduce the general case to the latter in step 2.

{\bf 1.} Assume that $\ind(A_1) = \ind(A_2) = 0$ and $\Delta(R(A_1),R(A_2)) = 0$. Without loss of generality, assume furthermore that $\sigma_1\le\sigma_2$. For $j=1,2$, let $U_j\in L(\calH,\calK)$ be a unitary operator satisfying $A_j = U_j|A_j|$, $U_jN(A_j) = R(A_j)^\perp$, $U_j|N(A_j)^\perp = Q_j|N(A_j)^\perp$, and $U_j^*Q_j = P_{N(A_j)^\perp}$, whose existence is ensured by Lemma \ref{l:ext}. First of all, we set $X := U_2^*Q_1 - P_{N(A_2)^\perp}$ and observe that
\begin{align*}
\|Q_1 - Q_2\|
&= \|U_2^*(Q_1 - Q_2)\| = \|U_2^*Q_1 - P_{N(A_2)^\perp}\| = \|X\|.
\end{align*}
Now, define $D_j := |A_j| + \sigma_j P_{N(A_j)}$, $j=1,2$, and note that $\sigma(D_j)\subset [\sigma_j,\infty)$. Therefore $\sigma(D_1)\cap\sigma(-D_2) = \emptyset$ and $\dist(\sigma(D_1),\sigma(-D_2)) = \sigma_1 + \sigma_2$. Hence, Theorem \ref{t:sylvester} yields
$$
\|X\|\,\le\,\frac{\|XD_1 + D_2X\|}{\sigma_1 + \sigma_2},
$$
which we shall further estimate in the following. A simple computation gives
\begin{align}
\begin{split}\label{e:four}
XD_1 + D_2X
=~ &\big(U_2^*Q_1|A_1| - |A_2|\big) + \big(|A_2|U_2^*Q_1 - P_{N(A_2)^\perp}|A_1|\big)\\
&- \sigma_1 P_{N(A_2)^\perp}P_{N(A_1)} + \sigma_2 P_{N(A_2)}U_2^*U_1 P_{N(A_1)^\perp}.
\end{split}
\end{align}
In the sequel, we will show that the norm of each of these four summands is not larger than $\|A_1 - A_2\|$, which then yields the result. The first one is simple:
\begin{align*}
\left\|U_2^*Q_1|A_1| - |A_2|\right\| = \left\|Q_1|A_1| - U_2|A_2|\right\| = \|A_1 - A_2\|.
\end{align*}
The norm of the second summand in \eqref{e:four} is estimated as follows:
\begin{align*}
\left\||A_2|U_2^*Q_1 - P_{N(A_2)^\perp}|A_1|\right\|
&= \left\|P_{R(|A_2|)}\left(|A_2|U_2^*Q_1 - |A_1|\right)\right\|\,\le\,\left\||A_2|U_2^*Q_1 - |A_1|\right\|\\
&= \left\|Q_1^*U_2|A_2| - |A_1|\right\| = \left\|Q_1\left(Q_1^*U_2|A_2| - |A_1|\right)\right\|\\
&= \left\|P_{R(Q_1)}U_2|A_2| - Q_1|A_1|\right\|\,\le\,\left\|U_2|A_2| - Q_1|A_1|\right\|\\
&= \|A_1 - A_2\|.
\end{align*}
In the estimation of the third summand in \eqref{e:four} we make use of Lemma \ref{l:ranclosed}.
\begin{align*}
\left\|\sigma_1 P_{N(A_2)^\perp}P_{N(A_1)}\right\|
&\le \sigma_1\left\|\left(A_2\big|N(A_2)^\perp\right)^{-1}\right\|\left\|A_2P_{N(A_2)^\perp}P_{N(A_1)}\right\| = \frac{\sigma_1}{\sigma_2}\left\|A_2P_{N(A_1)}\right\|\\
&\le\|(A_2-A_1)P_{N(A_1)}\|\,\le\,\|A_1 - A_2\|.
%
%&=\sigma_1\left\|P_{N(A_2)^\perp}\big|N(A_1)\right\| = \sigma_1\,\delta(N(A_1),N(A_2))\\
%&= \sigma_1\,\delta(N(A_2),N(A_1)) = \sigma_1\left\|P_{N(A_1)^\perp}P_{N(A_2)}\right\|\\
%&\le\sigma_1\left\|\left(A_1\big|N(A_1)^\perp\right)^{-1}\right\|\left\|A_1P_{N(A_1)^\perp}P_{N(A_2)}\right\|\\
%&= \left\|A_1P_{N(A_2)}\right\| = \left\|\left(A_1 - A_2\right)P_{N(A_2)}\right\|\le \left\|A_1 - A_2\right\|.
\end{align*}
In the estimate of the last term in \eqref{e:four} we use the fact that for a closed subspace $V$ and a unitary operator $U$ we have $P_{UV} = UP_VU^*$ and $(UV)^\perp = UV^\perp$. In addition, we will use $\Delta(R(A_1),R(A_2)) = 0$.
\begin{align*}
\left\|\sigma_2 P_{N(A_2)}U_2^*U_1P_{N(A_1)^\perp}\right\|
&=\sigma_2\left\|U_2P_{N(A_2)}U_2^*U_1P_{N(A_1)^\perp}U_1^*\right\| = \sigma_2\left\|P_{U_2N(A_2)}P_{U_1N(A_1)^\perp}\right\|\\
&=\sigma_2\left\|P_{R(A_2)^\perp}P_{R(A_1)}\right\| = \sigma_2\,\delta(R(A_1),R(A_2))\\
&=\sigma_2\,\delta(R(A_2),R(A_1)) = \sigma_2\left\|P_{R(A_1)^\perp}P_{R(A_2)}\right\|\\
%&=\sigma_2\left\|P_{U_1N(A_1)}P_{U_2N(A_2)^\perp}\right\| = \sigma_2\left\|U_1P_{N(A_1)}U_1^*U_2P_{N(A_2)^\perp}U_2^*\right\|\\
&=\sigma_2\left\|P_{N(A_1)}U_1^*U_2P_{N(A_2)^\perp}\right\| = \sigma_2\left\|P_{N(A_2)^\perp}U_2^*U_1 P_{N(A_1)}\right\|\\
&\le\sigma_2\left\|\left(|A_2|\big|N(A_2)^\perp\right)^{-1}\right\|\left\||A_2|P_{N(A_2)^\perp}U_2^*U_1 P_{N(A_1)}\right\|\\
&= \left\||A_2|U_2^*U_1 P_{N(A_1)}\right\| = \left\|\left(|A_2|U_2^*U_1 - |A_1|\right)P_{N(A_1)}\right\|\\
&\le\left\||A_2|U_2^*U_1 - |A_1|\right\| = \left\||A_2|U_2^* - |A_1|U_1^*\right\|\\
&= \left\|U_2|A_2| - U_1|A_1|\right\| = \|A_1 - A_2\|.
\end{align*}
We have just proved that the assumptions $\ind(A_1) = \ind(A_2) = 0$, $\Delta(R(A_1),R(A_2)) = 0$, and $\sigma_1\le\sigma_2$ lead to \eqref{e:main_estimate}. Now, assume that $\Delta(N(A_1),N(A_2)) = 0$ holds instead of $\Delta(R(A_1),R(A_2)) = 0$. Then $\Delta(R(A_1^*),R(A_2^*)) = 0$ (cf.\ \eqref{e:gap_proj2}). And since also $\ind(A_1^*) = \ind(A_2^*) = 0$ as well as $\sigma_{A_j^*} = \sigma_{A_j}$ (see Lemmas \ref{l:ranclosed} and \ref{l:swap}), we can apply what we just proved to the pair $(A_1^*,A_2^*)$ and obtain
\begin{equation}\label{e:duality}
\left\|Q_{A_1^*} - Q_{A_2^*}\right\|\,\le\,\frac{4}{\sigma_1 + \sigma_2}\,\|A_1^* - A_2^*\|.
\end{equation}
As $Q_{A_j^*} = Q_j^*$ (see Lemma \ref{l:polar1}), we again obtain \eqref{e:main_estimate}. This proves the theorem for the case $\ind(A_1)=\ind(A_2)=0$.

{\bf 2.} Assume that $\ind(A_1) = \ind(A_2) < 0$. Then, in particular, $\nuli(A_j) < \defi(A_j)$ for $j=1,2$. We have to distinguish between the cases where $\defi(A_j)$ is finite or infinite. From $\ind(A_1) = \ind(A_2)$ it follows that $\defi(A_1)$ is finite if and only if $\defi(A_2)$ is. If these numbers are infinite, then $\defi(A_1) = -\ind(A_1) = -\ind(A_2) = \defi(A_2)$. We set $m := -\ind(A_j)$, which is a finite or infinite cardinal number. Choose any Hilbert space $\calF$ with $\dim\calF = m$, let $\tilde{\calH} = \calH\oplus\calF$, and define the extension $\tilde A_j\in L(\tilde{\calH},\calK)$ of $A_j$ by
$$
\tilde A_j(x,y) = A_j x, \quad x\in\calH,\;y\in\calF,\;j=1,2.
$$
Then $R(\tilde A_j) = R(A_j)$ ($j=1,2$) is closed in $\calK$ and we have
\begin{align*}
N(\tilde A_1) = N(A_1)\oplus\calF
\qquad\text{and}\qquad
N(\tilde A_2) = N(A_2)\oplus\calF.
\end{align*}
Thus, for $j=1,2$ we obtain that $\nuli(\tilde A_j) = \nuli(A_j) + m = \defi(A_j)$ if $m$ is finite. But also if $m$ is infinite, we have $\nuli(\tilde A_j) = \dim(N(A_j)\oplus\calF) = \nuli(A_j) + m = m = \defi(A_j)$. Therefore, $\ind(\tilde A_1) = \ind(\tilde A_2) = 0$. Moreover $\Delta(R(\tilde A_1),R(\tilde A_2)) = \Delta(R(A_1),R(A_2))$ and (as is easily verified)
$$
\Delta(N(\tilde A_1),N(\tilde A_2)) = \Delta(N(A_1)\oplus\calF,N(A_2)\oplus\calF) = \Delta(N(A_1),N(A_2)).
$$
By the first part of the proof, 
\[
\|\tilde Q_1 - \tilde Q_2\|\,\le\,\frac{4}{\tilde \sigma_1 + \tilde \sigma_2}\,\|\tilde A_1 - \tilde A_2\|,
\]
where $\tilde Q_j := Q_{\tilde A_j}$ is the angular factor of $\tilde A_j$ and $\tilde\sigma_j := \sigma_{\tilde A_j}$, $j=1,2$. By definition of $\tilde A_1$ and $\tilde A_2$, it is clear that $\|\tilde A_1 - \tilde A_2\| = \| A_1 -  A_2\|$. With respect to the decomposition $\tilde\calH = \calH\oplus\calF$, the operator $\tilde A_j$, $j\in\{1,2\}$, has the matrix representation $\tilde A_j = [A_j,0]$. Hence,
$$
|\tilde A_j| = \left[\begin{matrix}|A_j| & 0\\0 & 0\end{matrix}\right].
$$
This immediately implies $\sigma(|\tilde A_j|) = \sigma(|A_j|)\cup\{0\}$ and thus $\sigma_{\tilde A_j} = \sigma_{A_j} = \sigma_j$. Since $[Q_j,0]$ is a partial isometry with kernel $N(\tilde A_j)$ and $\tilde A_j = [Q_j,0]|\tilde A_j|$, we have $\tilde Q_j = [Q_j,0]$ and thus $\|\tilde Q_1 - \tilde Q_2\| = \| Q_1 -  Q_2\|$. This completes the proof for the case $\ind (A_1) = \ind (A_2) < 0$. In the case $\ind(A_1) = \ind(A_2) > 0$ we have $\ind(A_1^*) = \ind(A_2^*) < 0$ as well as $\Delta(N(A_1^*),N(A_2^*)) = 0$ or $\Delta(R(A_1^*),R(A_2^*)) = 0$. This gives \eqref{e:duality} and thus \eqref{e:main_estimate}.
\end{proof}

Theorem \ref{t:fredholm} is now an easy consequence of Theorem \ref{t:main}.

\begin{proof}[Proof of Theorem \ref{t:fredholm}]
It is clear that $n := \nuli(A_1) = \nuli(A_2)$ and $d := \defi(A_1) = \defi(A_2)$ imply $\ind(A_1) = \ind(A_2)$. Moreover, $n < \infty$ or $d < \infty$, together with Corollary \ref{c:suff_Deltazero}, yields $\Delta(N(A_1),N(A_2)) = 0$ or $\Delta(R(A_1),R(A_2)) = 0$. Thus, the conditions in Theorem \ref{t:main} are met.
\end{proof}

For $\la_0\in\C$ we set $\dot{U}_{\veps}(\la_0) := \{\la\in\C : 0<|\la-\la_0|<\veps\}$. The following corollary is an immediate consequence of Theorem \ref{t:fredholm} and the well known ``punctured neighborhood theorem'' (cf.\ \cite[Theorem III.18.7]{m}).

\begin{cor}
Let $A\in L(\calH)$ be semi-Fredholm. Then there exists $\veps>0$ such that the mapping $\dot{U}_{\veps}(0)\to L(\calH)$, $\la\mapsto Q_{A-\lambda I}$, is continuous.
\end{cor}

\begin{rem}\label{s:remark}
The perturbation bound for the angular factor in Theorem \ref{t:main} can be further improved just as in \cite{hmz} by Hong, Meng, and Zheng in the matrix case. Without loss of generality, we can again assume that $\ind(A_1) = \ind(A_2) = 0$, $\Delta(R(A_1),R(A_2)) = 0$ and $\sigma_1\le\sigma_2$. Denote the sum of the last three terms in \eqref{e:four} by $T$. Then
\begin{align*}
\|T\|^2
&=\left\|\left(|A_2|U_2^*Q_1 - P_{N(A_2)^\perp}|A_1|\right) + \sigma_2 P_{N(A_2)}U_2^*U_1 P_{N(A_1)^\perp} - \sigma_1 P_{N(A_2)^\perp}P_{N(A_1)}\right\|^2\\
&\le\left\|\left(|A_2|U_2^*Q_1 - P_{N(A_2)^\perp}|A_1|\right) - \sigma_1 P_{N(A_2)^\perp}P_{N(A_1)}\right\|^2 + \left\|\sigma_2 P_{N(A_2)}U_2^*U_1 P_{N(A_1)^\perp}\right\|^2\\
&=\left\|\left(Q_1^*U_2|A_2| - |A_1|P_{N(A_2)^\perp}\right) - \sigma_1 P_{N(A_1)}P_{N(A_2)^\perp}\right\|^2 + \left\|\sigma_2 P_{N(A_1)^\perp}U_1^*U_2P_{N(A_2)}\right\|^2\\
&\le\left\||A_2|U_2^*Q_1 - P_{N(A_2)^\perp}|A_1|\right\|^2 + \left\|\sigma_1 P_{N(A_2)^\perp}P_{N(A_1)}\right\|^2 + \left\|\sigma_2 P_{N(A_1)^\perp}U_1^*U_2P_{N(A_2)}\right\|^2.
\end{align*}
A short look into the proof of Theorem \ref{t:main} reveals that the first and the third summand can be estimated by $\|A_1-A_2\|^2$ and that $\|\sigma_1 P_{N(A_2)^\perp}P_{N(A_1)}\|\le\tfrac{\sigma_1}{\sigma_2}\|A_1-A_2\|$. Thus, we obtain $\|T\|^2\le C\|A_1 - A_2\|^2$, where
$$
C = 2 + \frac{\sigma_1^2}{\sigma_2^2} = 1 + \frac{\sigma_1^2 + \sigma_2^2}{\max\{\sigma_1^2,\sigma_2^2\}}.
$$
%By the proof of Theorem \ref{t:main} we have $\|\sigma_2 P_{N(A_1)^\perp}U_1^*U_2P_{N(A_2)}\|\le\|A_1-A_2\|$. But also
%\begin{align*}
%\left\|\sigma_2 P_{N(A_1)^\perp}U_1^*U_2P_{N(A_2)}\right\| 
%&\le\sigma_2\left\|\left(|A_1|\big|N(A_1)^\perp\right)^{-1}\right\|\left\| |A_1|P_{N(A_1)^\perp}U_1^*U_2 P_{N(A_2)}\right\|\\
%&\le \frac{\sigma_2}{\sigma_1} \|A_1 - A_2\|.
%\end{align*}
%Thus, we have
%$$
%\|\sigma_2 P_{N(A_1)^\perp}U_1^*U_2P_{N(A_2)}\|\,\le\,\min\left\{1,\frac{\sigma_2}{\sigma_1}\right\}\|A_1 - A_2\| = \frac{\sigma_2}{\max\{\sigma_1,\sigma_2\}}\,\|A_1-A_2\|.
%$$
%Similarly, one sees that $\|\sigma_1 P_{N(A_2)^\perp}P_{N(A_1)}\|\le\tfrac{\sigma_1}{\max\{\sigma_1,\sigma_2\}}\|A_1-A_2\|$.
Hence, the following improved perturbation bound holds.
\end{rem}

\begin{thm}
Let $\calH$ and $\calK$ be Hilbert spaces and let $A_1,A_2\in L(\calH,\calK)\setminus\{0\}$ satisfy the conditions in Theorem \rmref{t:main}. If $A_j = Q_j|A_j|$ is the polar decomposition of $A_j$, $j=1,2$, then
\begin{equation*}%\label{e:hmz_estimate}
\|Q_1 - Q_2\|\,\le\,\frac{1}{\sigma_1 + \sigma_2}\, \left(1+\sqrt{1+\frac{\sigma_1^2+\sigma_2^2}{\max\{\sigma_1^2, \sigma_2^2\}}}\,\right)\|A_1 - A_2\|,
\end{equation*}
where $\sigma_j$ denotes the smallest positive spectral value of $|A_j| = (A_j^*A_j)^{1/2}$, $j=1,2$.
\end{thm}

We conclude this paper by proving Theorem \ref{t:cr_gap}. The main part of its proof is already covered by the arguments in the proof of Theorem \ref{t:main}.

\begin{proof}[Proof of Theorem \rmref{t:cr_gap}]
Without loss of generality, let $\sigma_2\geq\sigma_1$. Set $X = Q_2^*Q_1 - P_{N(A_2)^\perp}$ and observe that
\begin{align*}
\|Q_1 - Q_2\|
&\leq \|P_{N(Q_2^*)^\perp} Q_1 - Q_2\| + \|P_{N(Q_2^*)}Q_1\|\\
&= \|Q_2^*P_{N(Q_2^*)^\perp} Q_1 - P_{N(Q_2)^\perp}\| + \|P_{N(Q_2^*)}Q_1\|\\
&= \|X\| + \|P_{N(Q_2^*)}Q_1\|.
\end{align*}
Define $D_j := |A_j| + \sigma_j P_{N(A_j)}$, $j=1,2$. By Theorem \ref{t:sylvester},
$$
\|X\|\,\le\,\frac{\|XD_1 + D_2X\|}{\sigma_1 + \sigma_2}.
$$
We again compute
\begin{align}
\begin{split}\label{e:four2}
XD_1 + D_2X
=~ &\big(Q_2^*Q_1|A_1| - |A_2|\big) + \big(|A_2|Q_2^*Q_1 - P_{N(A_2)^\perp}|A_1|\big)\\
&+ \sigma_2 P_{N(A_2)}Q_2^*Q_1 P_{N(A_1)^\perp} - \sigma_1 P_{N(A_2)^\perp}P_{N(A_1)}.
\end{split}
\end{align}
From $P_{N(A_2)}Q_2^* = (Q_2P_{N(Q_2)})^* = 0$ it follows that $\sigma_2 P_{N(A_2)}Q_2^*Q_1 P_{N(A_1)^\perp} = 0$, and the second term in \eqref{e:four2} can be estimated exactly as in the proof of Theorem \ref{t:main}. As seen in Remark \ref{s:remark}, one has $\|\sigma_1 P_{N(A_2)^\perp}P_{N(A_1)}\|\le\frac{\sigma_1}{\sigma_2} \|A_1 - A_2\|\le\|A_1 - A_2\|$. Finally, the first term in (\ref{e:four2}) can be estimated as
\begin{align*}
\|Q_2^*Q_1|A_1| - |A_2|\| &= \|P_{R(Q_2)}Q_1|A_1| - Q_2|A_2|\| 
\leq \|Q_1|A_1| - Q_2|A_2|\| = \|A_1 - A_2\|.
\end{align*}
So, we have shown that
$$
\|X\| \leq \frac{3}{\sigma_1+\sigma_2}\|A_1-A_2\|.
$$
It remains to estimate $\|P_{N(Q_2^*)}Q_1\|$. For this, observe that (see \eqref{e:gap_proj2} and \eqref{e:gap_proj})
\begin{align*}
\|P_{N(Q_2^*)}Q_1\|
&=\|P_{N(Q_2^*)}Q_1 - P_{N(Q_1^*)}Q_1\|\le\|P_{N(Q_2^*)}- P_{N(Q_1^*)}\|\\
&= \|P_{R(A_2)^\perp} - P_{R(A_1)^\perp}\| = \hat\delta\left(R(A_2),R(A_1)\right).
\end{align*}
Now, the estimate \eqref{e:side_estimate} follows from the first part of Lemma \ref{l:gappyhelpy}. If $\hat\delta(R(A_1),R(A_2)) < 1$, then \eqref{e:mitgap} is a consequence of the second part of Lemma \ref{l:gappyhelpy}. Let $\hat\delta(N(A_1),N(A_2)) < 1$. Then $\hat\delta(R(A_1^*),R(A_2^*)) < 1$ and \eqref{e:mitgap} follows from what is already proved and the relations $Q_{A^*} = Q_A^*$ and $\sigma_{A^*} = \sigma_A$.
\end{proof}

\begin{rem}\label{r:smallpert}
Finally, as announced in the Introduction, let us show that Theorem \ref{t:cr_gap} generalizes the results in the finite-dimensional situation if $A_1$ and $A_2$ are small perturbations of each other. More precisely, we show the following:
$$
\text{{\it If $\rank(A_1) = \rank(A_2)$ and $\|A_1 - A_2\| < \max\{\sigma_1,\sigma_2\}/3$, then $\hat\delta(R(A_1),R(A_2)) < 1$}.}
$$
For this, assume that, e.g., $\sigma_2\le\sigma_1$ and $\|A_1 - A_2\| < \sigma_1/3$. We have $\sigma_j = \|A_j^\dagger\|^{-1}$, where $A_j^\dagger$ denotes the Moore-Penrose inverse of $A_j$, $j=1,2$. By \cite[Theorem 4.1]{w},
$$
\|A_1^\dagger - A_2^\dagger\|\le 2\|A_1^\dagger\|\|A_2^\dagger\|\|A_1 - A_2\|.
$$
Therefore,
\begin{align*}
\sigma_1 - \sigma_2 = \frac 1{\|A_1^\dagger\|} - \frac 1{\|A_2^\dagger\|} = \frac{\|A_2^\dagger\| - \|A_1^\dagger\|}{\|A_1^\dagger\|\|A_2^\dagger\|}\le\frac{\|A_1^\dagger - A_2^\dagger\|}{\|A_1^\dagger\|\|A_2^\dagger\|}\le 2\|A_1 - A_2\|.
\end{align*}
Thus, Lemma \ref{l:gappyhelpy} implies $\delta(R(A_1),R(A_2))\le\|A_1 - A_2\|/\sigma_1 < 1/3 < 1$ and
$$
\delta(R(A_2),R(A_1))\le\frac{\|A_1 - A_2\|}{\sigma_2}\le\frac{\|A_1-A_2\|}{\sigma_1 - 2\|A_1-A_2\|} < 1.
$$
This proves $\hat\delta(R(A_1),R(A_2)) < 1$.
\end{rem}

\end{document}